\documentclass[10pt,reqno]{amsart}

\usepackage[margin=1.0in]{geometry}
\usepackage{amsmath, amsfonts, bbm, dsfont, stmaryrd}
\usepackage{amsthm}
\usepackage{amsopn}
\usepackage{amssymb}
\usepackage{physics}
\usepackage{dsfont}
\usepackage{cite}
\usepackage[colorlinks=true]{hyperref}
\usepackage{upgreek}

\newcommand{\beq}{ \begin{equation} }
\newcommand{\eeq}{ \end{equation} }

\newcommand*\di{\mathop{}\!\mathrm{d}}
\newcommand*\ii{\mathop{}\!\mathrm{i}}

\theoremstyle{plain}
\newtheorem{theorem}{Theorem}[section]
\numberwithin{theorem}{section}
\newtheorem{lemma}[theorem]{Lemma}
\newtheorem{proposition}[theorem]{Proposition}

\newtheorem{corollary}[theorem]{Corollary}

\theoremstyle{definition}
\newtheorem{definition}[theorem]{Definition}
\newtheorem{assumption}[theorem]{Assumption}
\newtheorem{remark}[theorem]{Remark}

\numberwithin{equation}{section}

\begin{document}

\title{Local law and Tracy--Widom limit for sparse sample covariance matrices}

\author{Jong Yun Hwang}
\address{Department of Mathematical Sciences, KAIST, Daejeon, Korea}
\email{jyh9006@kaist.ac.kr}

\author{Ji Oon Lee}
\address{Department of Mathematical Sciences, KAIST, Daejeon, Korea}
\email{jioon.lee@kaist.edu}

\author{Kevin Schnelli}
\address{Department of Mathematics, KTH Royal Institute of Technology, Stockholm, Sweden}
\email{schnelli@kth.se}


\maketitle
\begin{abstract}
We consider spectral properties of sparse sample covariance matrices, which includes biadjacency matrices of the bipartite Erd\H{o}s--R\'enyi graph model. We prove a local law for the eigenvalue density up to the upper spectral edge. Under a suitable condition on the sparsity, we also prove that the limiting distribution of the rescaled, shifted extremal eigenvalues
is given by the GOE Tracy--Widom law with an explicit formula on the deterministic shift of the spectral edge. For the biadjacency matrix of an Erd\H{o}s--R\'enyi graph with two vertex sets of comparable sizes $M$ and $N$, this establishes Tracy--Widom fluctuations of the second largest eigenvalue when the connection probability $p$ is much larger than $N^{-2/3}$ with a deterministic shift of order $(Np)^{-1}$.
\end{abstract}

\vskip20pt

\section{Introduction}

Sample covariance matrices form one fundamental class of random matrices. They are of great importance in high-dimensional data and multivariate statistics. Spectral properties of sample covariance matrices are particularly interesting as they are crucial in various procedures in real data analysis such as principal component analysis (PCA). 

It is well known that the limiting spectral distribution of a sample covariance matrix is governed by the Marchenko--Pastur law when the sample size is comparable to the dimension \cite{MarchenkoPastur1967}. In this case, the fluctuations of the largest eigenvalue of the sample covariance matrix follow the Tracy--Widom laws. This was first proved for the complex Wishart ensemble by Johansson \cite{Johansson2000} and for the real Wishart ensemble by Johnstone \cite{Johnstone2001}. For sample covariance matrices with non-Gaussian entries but identity population covariance, the Tracy--Widom limit for the largest eigenvalue was established by Pillai and Yin \cite{PillaiYin2014}, and a necessary and sufficient condition on the entries' distribution for the limit to hold was obtained by Ding and Yang \cite{DingYang2017}. In the general non-null case, where the general population matrix is not a multiple of the identity matrix, the Tracy--Widom limit was identified in \cite{ElKaroui2007,LeeSchnelli2016} and the edge universality was proved in \cite{BaoPanZhou2015,KnowlesYin2017}. 

The Tracy--Widom laws are also the answer to the question about the fluctuations of the largest eigenvalues of various random matrix models, including the adjacency matrix of random graphs. The simplest case is the Erd\H{o}s--R\'enyi graph, where any pair of vertices are connected with probability $p$, independently from other pairs. For the adjacency matrices of the Erd\H{o}s--R\'enyi graph, while the largest eigenvalue is macroscopically separated from the bulk of the spectrum, the second largest eigenvalue exhibits Tracy--Widom fluctuations as long as the connection probability~$p$ is independent of the size $N$. We remark that the fluctuations of the second largest eigenvalue encode key properties of the random graph in many applications, including hypothesis testing for community detection in stochastic block model \cite{BickelSarkar2016,Lei2016}.

The study of the behavior of the second largest eigenvalue becomes significantly harder when $p\equiv p(N)$ scales as $N$ varies. If $p \to 0$ as $N \to \infty$, we say that the graph (or the corresponding adjacency matrix) is sparse. The analysis for the sparse case was only recently done by Erd\H{o}s, Knowles, Yau and Yin \cite{ErdosKnowlesYauYin2013,ErdosKnowlesYauYin2012} and was further extended in \cite{LeeSchnelli2016} and \cite{HuangLandonYau2017}. The main idea in a deeper analysis of the local eigenvalue density, especially for sparse matrices, is the local semicircle law. It is the further local refinement of Wigner's semicircle law that governs the global behavior of the eigenvalue density of Wigner matrices. 

If the underlying random graph is directed, we lose the symmetry of the adjacency matrix, and we are led to consider the singular values of the matrices instead of the eigenvalues. Related to such a model is a bipartite random graph, where the vertices can be decomposed into two groups within which vertices are not connected to each other. In these cases, the adjacency matrices can naturally be identified with sample covariance matrices, and when the connection probability $p$ tends to zero as the size of matrix grows, we are led to consider sparse sample covariance matrices.

Another motivation for our research on the largest eigenvalues of sample covariance matrices stems from PCA in statistics and signal processing. In practice, one of the most significant challenges in using PCA is to determine the number of components, i.e., the rank of the signal matrix in a noisy matrix. For the null hypothesis testing, which tests whether a signal is present, the largest eigenvalues can be exploited for the test statistics as discussed in \cite{NadakuditiEdelman2008,BDMN2008}. The idea can be further extended to the estimation of the true rank by a sequential application of the strategy in \cite{KritchmanNadler2008} or by a method based on the conditional singular value test \cite{ChoiTaylorTibshirani2017}. We refer to \cite{BaoPanZhou2015} for more applications of the largest eigenvalues in high-dimensional statistical inference.

In this paper, we study the spectral properties of sparse sample covariance matrices, including the local law for $p \gg N^{-1}$ and the behavior of the largest eigenvalues. We prove that the limiting distribution of the shifted, rescaled largest eigenvalue is given by GOE Tracy--Widom distribution if $p \gg N^{-2/3}$, which was also assumed in \cite{LeeSchnelli2017}. The shift is deterministic, and it can be precisely computed in terms of $N$ and $p$. As in \cite{LeeSchnelli2017}, while the shift for the case $p \gg N^{-1/3}$ is negligible, it is larger than the Tracy--Widom fluctuations if $p \ll N^{-1/3}$. 

Our analysis is based on a refined local Marchenko--Pastur law that provides estimates for the local eigenvalue density down to the optimal scale at the upper edge of the spectrum.
Following the strategy in \cite{LeeSchnelli2016}, we introduce a polynomial $P$ that describes a deterministic correction to Marcenko-Pastur law. These corrections provide the crucial information about the precise location of the largest eigenvalue. Adapting the approach in \cite{LeeSchnelli2016}, we obtain the polynomial by using resolvent expansions instead of the Schur complement formula, but due to the structure of the sample covariance matrix, we employ a linearization trick before applying resolvent expansion methods. Technically, we control high moments of $|P|$ by a recursive moment estimate, where the bound utilizes the lower moments of $|P|$. After establishing the local Marchenko--Pastur law, we use the Green function comparison method to prove the Tracy--Widom limit of the largest eigenvalue. In the Green function comparison, instead of applying the Lindberg swapping trick with matching moments, we use a continuous flow that interpolates the given sparse sample covariance matrix and a `nicely-behaving' sample covariance matrix, which is a Wishart ensemble. The continuous interpolation has been successfully used in the context of edge universality in various random matrix models;
we refer to Section \ref{subsec:outline} for further details. Another technical difficulty stems from the lack of symmetry in Marchenko--Pastur law. Compared to the Wigner case in \cite{LeeSchnelli2016}, the polynomial $P$ has a more complicated form, and it makes the analysis for the limiting distribution more involved; see Section \ref{sec:stieltjes} for technical details.

For sparse matrices of Wigner-type, the result in \cite{LeeSchnelli2016} was extended by Huang, Landon and Yau \cite{HuangLandonYau2017} to cover the regime $N^{-7/9} \ll p \leq N^{-2/3}$. In that case they proved that the shifted, rescaled largest eigenvalue exhibits a phase transition from Tracy--Widom to Gaussian fluctuations at $p \sim N^{-2/3}$. We expect an analogous transition to occur for sparse sample covariance matrices, but we do not pursue this direction in the current article.

This paper is organized as follows: In Section \ref{sec:prelim}, we define the model and introduce the main results with the outline of the strategy for the proof. In Section \ref{sec:stieltjes}, we prove some properties of the deterministic refinement of the Marchenko--Pastur law. In Section \ref{sec:local}, we prove the local law based on the key technical result on the recursive moment estimates. In Section \ref{sec:TW}, we prove the Tracy--Widom limit of the largest eigenvalue via the Green function comparison method. Some technical results are adaptions from~\cite{LeeSchnelli2016} and are therefore postponed to the Supplement~\cite{supplement}.

\begin{remark}[Notational remark 1]
We use the symbols $O( \cdot )$ and $o( \cdot )$ for the standard big-O and little-o notation. The notations $O$, $o$, $\ll$, $\gg$ refer to the limit $N,M \to \infty$ with $N/M=d$ fixed unless otherwise stated. The notation $a \ll b$ means $a = o(b)$. We use $c$ and $C$ to denote positive constants that do not depend on $N$, usually with the convention $c \leq C$. Their values may change from line to line. We write $a \sim b$ if there is $C \geq 1$ such that $C^{-1}|b| \leq |a| \leq C|b|$.
\end{remark}

\section{The model and main results}\label{sec:prelim}

\subsection{Motivating examples} 

Before presenting our model in detail and stating our main results, we outline a few motivating examples for the present work.

\subsubsection{Signal detection with missing values} \label{subsec:example1}

Consider an $M$-dimensional signal-plus-noise vector
\begin{align}
	\mathbf{y}:= A \mathbf{s} + \mathbf{z},
\end{align}
where $\mathbf{s}$ is an $N$-dimensional signal vector, $A$ is an $M \times N$ deterministic matrix, and $\mathbf{z}$ is an $M$-dimensional random vector. To understand the case where the coordinates $z_i$ $(1 \leq i \leq M)$ are mostly zero, hence $\mathbf{z}$ is sparse, we let
\begin{align}
z_i = B_i V_i,
\end{align}
where $(B_i)$ and $(V_i)$ are independent families of independent, identically distributed (i.i.d.) centered random variables. The random variables $(V_i)$ represent the noise and satisfy $\mathbb{E} V_i^2 = 1$ and $\mathbb{E} V_i^{2k} \leq C_k$ for some constants $(C_k)$. The random variables $(B_i)$ are of Bernoulli type with parameter $p$, i.e.,
\begin{align} \label{def:Bernoulli}
\mathbb{P} \Big(B_i = \frac{1}{\sqrt{Np}} \Big) =p, \qquad \mathbb{P}(B_i = 0) = 1-p.
\end{align}
The sparsity of the noise, which is determined by the probability $p$, may originate from the existence of a certain error threshold or the lack of data due to missing observations.

The first step in the analysis is to determine whether there is any signal present. Under the null hypothesis, the sample covariance matrix associated with $\mathbf{y}$ is nothing more than the one associated with $\mathbf{z}$. The limiting distribution of its largest eigenvalue is assumed to be determined by Tracy--Widom--Airy statistics if $p \sim 1$. One can also use the statistics of Onatski \cite{Onatski2009} given by $R:=(\mu_1-\mu_2)/(\mu_2-\mu_3)$, where $\mu_1, \mu_2, \mu_3$ denote the first, second, and the third largest eigenvalue, respectively. It was shown \cite{Onatski2009} in the complex setting that $R$ is asymptotically pivotal under the null hypothesis.

We introduce the sparsity parameter $q$ through
\begin{align}
p = \frac{q^2}{N}
\end{align}
with $0 < q = \sqrt{Np} \leq N^{1/2}$, where we allow $q$ to depend on $N$. When $q = N^{1/2}$, we regain a usual sample covariance matrix.

The concept of the sparsity is crucial in the analysis for small $N$. For example, while the case with probability $p=0.3$ would not be classified as a sparse case for large $N$, with $N=30$ the sparsity parameter $q = 9 = N^{\phi}$ for $\phi \approx 0.323$, which falls into the regime we are interested in.

\subsubsection{Simple Markov switching model} \label{subsec:example3}

The simplest Markov switching model involves a random vector
\begin{align}
	\mathbf{y}:= \mathbf{\mu}_{\mathbf{s}} + \sigma_{\mathbf{s}} \mathbf{z},
\end{align}
where $\mathbf{s}$ denotes the unobserved state indicator; see~\emph{e.g.} Equation (1) in \cite{Timmermann2000}. If the entries of $\mathbf{s}$ are i.i.d. Bernoulli random variables, it is called the random two-regime model; see~\emph{e.g.} \cite{Quandt1972}. Under the null hypothesis
\begin{align}
	\mathbf{H}_0 : \mathbf{\mu}_{\mathbf{s}} = \mathbf{0},
\end{align}
this model reduces to the one in Subsection \ref{subsec:example1}, if we simply write
\begin{align}
	\sigma_{\mathbf{s}} z_i = \sigma B_i V_i,
\end{align}
where $B_i$ is a Bernoulli random variable defined in \eqref{def:Bernoulli}. To generalize it further, one can choose a Markov chain with states $\{ 0, 1\}$ instead of a Bernoulli variable. Then, one can apply the results on the largest eigenvalues of sparse sample covariance matrices by analyzing the $2 \times 2$ transition matrix.

\subsubsection{Bipartite stochastic block model} \label{subsec:example2}

Consider a graph with two vertex sets $V_1$ and $V_2$ of size $M$ and $N$, respectively. Each $V_i$ is further divided into two or more communities. In the simple case of two communities in each vertex set, each $V_i$ ($i=1, 2$) is with partition $(P_i, Q_i)$, and edges are added independently at random between $V_1$ and $V_2$ with probabilities depending on which parts the vertices are in; edges between $P_1$ and $P_2$ or $Q_1$ and $Q_2$ are added with probability $p$, while the other edges are added with probability $p'$. This model was proposed by Feldman, Perkins, and Vempala \cite{Feldman2015NIPS} as a generalization of the classic stochastic block model to unify graph partitioning and planted constraint satisfaction problems (CSP's) into one problem. The spectral analysis on the model was studied in \cite{Florescu2016COLT}.

In the spectral analysis for the graph partitioning, the $M \times N$ biadjacency matrix $X$ can be considered, which is defined as follows: 
\begin{align} \label{eq:bsbm}
	X_{\alpha i} =
		\begin{cases}
		1 & \text{ if } v_{\alpha} \in V_1 \text{ and } w_i \in V_2 \text{ are connected}, \\
		0 & \text{ otherwise}.
		\end{cases}
\end{align}
Then, the singular vector of $X$ corresponding to the second largest singular value, or equivalently, the second eigenvector $X X^{\dag}$ is correlated with the partition of $V_1$. Such an algorithm requires that the second largest eigenvalue is well-separated from the spectral norm of the noise matrix. If the probability $p$ and $p'$ tend to $0$ as $M, N \to \infty$, the graph is sparse and the noise matrix becomes a sparse sample covariance matrix. The null model in this case is the biadjacency matrix of the bipartite Erd\H{o}s--R\'enyi graph where $p=p'$.

\subsection{Definitions and notations}
We consider the following type of sample covariance matrices:
\begin{definition}[Sample covariance matrices] Let $X=(X_{\alpha i})$ be an real $M\times N$ matrix with independent entries satisfying the moment conditions
\begin{align}
\mathbb{E}X_{\alpha i} =0,\quad \mathbb{E}(X_{\alpha i})^2 = \frac{1}{N}.
\end{align}
The sample covariance matrix associated with $X$ is given by $S=X^\dagger X$. Furthermore, $M\equiv M(N)$ with
\begin{align}\label{le d}
d_N=\frac{N}{M}\rightarrow d\in [1,\infty),
\end{align}
as $N\rightarrow\infty$. For simplicity, we assume that $d_N$ is constant, hence we use $d$ instead of $d_N$. 
\end{definition}

The assumption $M\equiv M(N)$ with $N/M \to d$ is reasonable in high-dimensional settings. We remark that our results also hold for the case $d < 1$ after taking $(M-N)$ zero eigenvalues into consideration.

Next, we introduce some basic definitions and set notations.

\begin{definition}[High probability event] We say that an $N$-dependent event $\Xi \equiv \Xi^{(N)}$ holds with high probability if, for any large $D>0$,
\begin{align}
\mathbb{P}(\Xi^{(N)})\geq 1-N^{-D},
\end{align}
for sufficiently large $N\geq N_0(D)$.
\end{definition}
\begin{definition}[Stochastic domination] Let $Y_1 \equiv Y_1^{(N)}, Y_2 \equiv Y_2^{(N)}$ be $N$-dependent non-negative random variables. We say that $Y_1$ stochastically dominates $Y_2$ if, for all small $\epsilon>0$ and large $D>0$,
\begin{align}
\mathbb{P}(Y_1^{(N)}>N^\epsilon Y_2^{(N)})\leq N^{-D},
\end{align}
for sufficiently large $N\geq N_0(\epsilon,D)$, and we write $Y_1\prec Y_2$.
\end{definition}

\begin{definition}[Stieltjes transform] Given a probability measure $\nu$ on $\mathbb{R}$, its Stieltjes transform is the analytic function $m_\nu : \mathbb{C}^+ \rightarrow \mathbb{C}^+$, with $\mathbb{C} ^+ := \{z=E+\ii\eta : E\in \mathbb{R}, \eta>0\}$, defined by
\begin{align}
m_\nu(z) := \int_{\mathbb{R}}\frac{\di\nu(x)}{x-z}, \quad (z\in\mathbb{C}^+).
\end{align}
\end{definition}

Note that $\lim_{\eta\rightarrow \infty} \ii\eta \, m_{\nu} (\ii\eta) = -1$ since $\nu$ is a probability measure. Conversely, if an analytic function $m:\mathbb{C}^+ \rightarrow \mathbb{C}^+$ satisfies $\lim_{\eta\rightarrow \infty} \ii\eta \, m(\ii\eta) = -1$, then it is the Stieltjes transform of a probability measure; see \emph{e.g.}~\cite{Akhiezer}.

Choosing $\nu$ to be the Marchenko--Pastur law with density 
\begin{align*}
\frac{1}{2\pi x}\sqrt{(\lambda_+ -x)(x-\lambda_-)}
\end{align*}
 on $[\lambda_-, \lambda_+]$, where 
\begin{align}
\lambda_+=\Big(1+\frac{1}{\sqrt d}\Big)^2, \quad \lambda_-=\Big(1-\frac{1}{\sqrt d}\Big)^2,
\end{align}
one can shows that $m_\nu$, denoted by $m_{\mathrm{MP}}$, is explicitly given by
\begin{align}\label{le MP}
m_{\mathrm{MP}}(z) = \frac{-\big(z+1-\frac{1}{d}\big) + \sqrt{(z+1-\frac{1}{d}\big)^2-4z}}{2z},
\end{align}
where we choose the branch of the square root such that $m_{\mathrm{MP}}(z) \in \mathbb{C}^+, z\in\mathbb{C}^+$. It directly follows that
\begin{align}
1+\big(z+1-\frac{1}{d}\big)m_{\mathrm{MP}}(z) + zm_{\mathrm{MP}}(z)^2 = 0, \quad (z\in\mathbb{C}^+).
\end{align}

\begin{definition}[Green function and its normalized trace] Given a real symmetric matrix $H$ we define its Green function by
\begin{align}
G^H(z):=\frac{1}{H-zI},\quad (z\in \mathbb{C}^+),
\end{align}
and the normalized trace of $G^H$ by
\begin{align}
m^H(z):=\frac{1}{N}\Tr G^H (z), \quad (z\in\mathbb{C}^+).
\end{align}
The matrix entries of $G^{H}(z)$ are denoted by $G^{H}_{ij}(z)$. In the following we often drop the explicit $z$-dependence from the notation for $G^H(z)$ and $m^H(z)$.
\end{definition}

\subsection{Main results}

The sparse sample covariance matrices we consider satisfies the following assumption:

\begin{assumption} \label{assume} Fix any small $\phi>0$. We assume that $X=(X_{\alpha i})$ is a real $M\times N$ matrix whose entries are independent, identically distributed (i.i.d.) random variables. We also assume that $(X_{\alpha i})$ satisfy the moment conditions
\begin{align}\label{eq:moments}
\mathbb{E}X_{\alpha i} =0,\quad \mathbb{E}(X_{\alpha i})^2 = \frac{1}{N}, \quad \mathbb{E}|X_{\alpha i}|^k \leq \frac{(Ck)^{ck}}{Nq^{k-2}},\quad (k\geq 3),
\end{align}
with sparsity parameter $q$ satisfying
\begin{align} \label{eq:sparsity}
N^{\phi}\leq q \leq N^{1/2}.
\end{align}
\end{assumption}

Note that the real Wishart ensemble corresponds to the case $\phi=1/2$. 
We denote by $\kappa^{(k)}$ the $k$-th cumulant of the i.i.d.\ random variables $(X_{\alpha i})$. Under Assumption \ref{assume} we have $\kappa^{(1)}=0, \kappa^{(2)}=1/N,$
\begin{align}
|\kappa^{(k)}|\leq \frac{(2Ck)^{2(c+1)k}}{Nq^{k-2}},\quad (k\geq 3).
\end{align}
For example, if $X_{\alpha i}$ is the Bernoulli random variable defined in \eqref{eq:bsbm} with $p=p'$, then letting $q = \sqrt{Np}$,
\begin{align}
\mathbb{E} \left[(X_{\alpha i})^k \right] = \frac{(-p)^k (1-p) + (1-p)^k p}{(Np(1-p))^{k/2}} = \frac{1}{Nq^{k-2}} (1+ O(p)) = \kappa^{(k)} (1+ O(p)).
\end{align}
We will also use the normalized cumulants, $s^{(k)}$, by setting
\begin{align}
s^{(1)}:=0, \quad s^{(2)}:=1, \quad s^{(k)}:=Nq^{k-2}\kappa^{(k)}, \quad (k\geq 3).
\end{align}

\subsubsection{Improved local law up to the upper edge for sparse random matrices} 

Let $\lambda_1 \geq \lambda_2 \geq \cdots \geq \lambda_N\geq 0$ be the ordered eigenvalues of $X^\dagger X$. Note that $m^{X^\dagger X}$ is the Stieltjes transform of the empirical eigenvalue distributions, $\mu^{X^\dagger X}$, of $X^\dagger X$ given by
\begin{align}
\mu^{X^\dagger X} := \frac{1}{N} \sum_{i=1}^N \delta_{\lambda_i} .
\end{align}
Recall $d$ from~\eqref{le d}. We introduce the following domain of the upper half-plane
\begin{align}
\mathcal{E} = \{ E+\ii \eta : \frac{\lambda-}{2} \leq E \leq \lambda_+ +1, \, 0 < \eta < 3 \} \qquad \text{if } d>1,
\end{align}
respectively,
\begin{align}
\mathcal{E} = \{ E+\ii \eta : \frac{1}{10} \leq E \leq \lambda_+ +1, \, 0 < \eta < 3 \} \qquad \text{if } d=1.
\end{align}

Our first main result is the local law for $m^{X^\dagger X}$ up to the upper spectral edge.

\begin{theorem}\label{thm:locallaw}
Let $X$ satisfy Assumption \ref{assume} with $\phi>0$. Then, there exist deterministic numbers $L_+ \geq L_- \geq 0$ and an algebraic function $\widetilde{m}:\mathbb{C}^+ \rightarrow \mathbb{C}^+$ such that the following hold:
\begin{enumerate}
\item[(1)] The function $\widetilde{m}$ is the Stieltjes transform of a deterministic probability measure $\widetilde{\rho}$, i.e., $\widetilde{m}(z) = m_{\widetilde{\rho}}(z).$ The measure $\widetilde\rho$ is supported on the interval $[L_-, L_+]$ and is absolutely continuous with respect to Lebesgue measure with a strictly positive, continuous density on $(L_-, L_+)$. 
\item[(2)] The function $\widetilde{m}\equiv\widetilde{m}(z)$, $z\in\mathbb{C}^+$, is a solution to the polynomial equation
\begin{align} \label{eq:polynomial}
P_z(\widetilde{m}):=1+\big(z+1-\frac{1}{d}\big)\widetilde{m}+z\widetilde{m}^2+\frac{s^{(4)}}{q^2}\widetilde{m}^2\big(z\widetilde{m}+1-\frac{1}{d}\big)^2 =0.
\end{align}
\item[(3)] The normalized trace $m^{X^\dagger X}$ of the Green function of $X^\dagger X$ satisfies the local law
\begin{align}\label{eq:locallaw}
|m^{X^\dagger X}(z) -\widetilde{m}(z)| \prec \frac{1}{q^2}+\frac{1}{N\eta},
\end{align}
uniformly on the domain $\mathcal{E}, z=E+\ii\eta$.
\end{enumerate}
The right endpoint $L_+$ is given by
\begin{align}\label{eq:L_+}
L_+=\Big(1+\frac{1}{\sqrt d}\Big)^2+\frac{1}{\sqrt d}\Big(1+\frac{1}{\sqrt d}\Big)^2 \frac{s^{(4)}}{q^2} +O(q^{-4})
\end{align}
and the left endpoint $L_-$ for the case $d > 1$ is given by
\begin{align}
L_-=\Big(1-\frac{1}{\sqrt d}\Big)^2-\frac{1}{\sqrt d}\Big(1-\frac{1}{\sqrt d}\Big)^2 \frac{s^{(4)}}{q^2} +O(q^{-4}).
\end{align}
\end{theorem}

The sparsity of the entries of $X$ makes its eigenvalues follow the deterministic law $\widetilde{\rho}$ that depends on the sparsity parameter $q$. While this law approaches the Marchenko--Pastur law,~$\rho_{\mathrm{MP}}$, when $N\rightarrow\infty$, its deterministic refinement to the standard Marchenko--Pastur law for finite~$N$ accounts for the non-optimality at the edge of previous results obtained in \cite{DingYang2017}. (See \eqref{eq:localMP1} below for the estimate in \cite{DingYang2017}.)

From the local law \eqref{eq:locallaw}, we can also obtain estimates on the local density of states of $X^\dagger X$. For $E_1 <E_2$, define
\begin{align}
\mathbf{n}(E_1, E_2):=\frac{1}{N}|\{i:E_1<\lambda_i\leq E_2\}|, \quad n_{\widetilde{\rho}}(E_1, E_2):=\int^{E_2}_{E_1}\widetilde{\rho}(x) \di x.
\end{align}
Applying the Helffer-Sj\"ostrand calculus with Theorem \ref{thm:locallaw}, we get the following corollary:

\begin{corollary}\label{cor:localdensity}
Suppose that $X$ satisfies Assumption \ref{assume} with $\phi>0$. Let $E_1, E_2\in\mathbb{R}, E_1<E_2$. Then,
\begin{align}
|\mathbf{n}(E_1, E_2)-n_{\widetilde{\rho}}(E_1, E_2)|\prec\frac{E_2-E_1}{q^2} +\frac{1}{N}.
\end{align} 
\end{corollary}

We can obtain the following estimates of the operator norm $\|X^\dagger X\|$ of $X^{\dagger}X$, by combining the local law with the deterministic refinement to the Marchenko--Pastur law. This allows us to assert a strong statement on the location of the extremal eigenvalues of $X^\dagger X$.
\begin{theorem} \label{thm:norm}
Suppose that $X$ satisfies Assumption \ref{assume} with $\phi>0$. Then,
\begin{align}
| \|X^\dagger X\| -L_+| \prec \frac{1}{q^4}+\frac{1}{N^{2/3}},
\end{align}
where $L_+$ is the right endpoint of the support of the measure $\widetilde{\rho}$ given in \eqref{eq:L_+}.
\end{theorem}

\subsubsection{Tracy--Widom limit of the largest eigenvalues}
Our second main result shows that the rescaled largest eigenvalues of the sparse sample covariance matrices of sparse random matrix converge in distribution to Tracy--Widom law, if the sparsity parameter $q$ satisfies $q\gg N^{1/6}$.
\begin{theorem}\label{thm:TWlimit}
Suppose that $H$ satisfies Assumption \ref{assume} with $\phi>1/6.$ Denote by $\lambda_1^{X^\dagger X}$ the largest eigenvalue of $X^\dagger X$. Then
\begin{align}
\lim_{N\rightarrow \infty}\mathbb{P}\Big( \gamma N^{2/3}\big(\lambda_1^{X^\dagger X} - L_+\big)\leq s\Big) = F_1(s)
\end{align}
where $L_+$ is given in \eqref{eq:L_+}, $\gamma$ is a constant defined as
\begin{align}
\gamma = d^{1/2} (1+\sqrt{d})^{-4/3},
\end{align}
and $F_1$ is the cumulative distribution function of the GOE Tracy--Widom law.
\end{theorem}

In the regime $q \leq N^{1/3}$, the deterministic shift of the right endpoint $L_+$ is crucial in the analysis of the largest eigenvalue, since the shift is of order $q^{-2}$, which is larger than $N^{-2/3}$, the scale of the Tracy--Widom fluctuation. Thus, to use the Tracy--Widom fluctuation of the largest eigenvalue in this regime, the correction from the fourth cumulant should be taken into consideration.

We further remark that all results in this subsection also hold for complex sparse sample covariance matrices with the GUE Tracy--Widom law as the governing law for the limiting edge fluctuation.

\subsection{Outline of proofs} \label{subsec:outline}
In the proof of the local law, we apply the strategy of cumulant expansion, which was used for sparse Wigner-type matrices in \cite{LeeSchnelli2016}. For the basic ideas of the cumulant expansion method in the proof of the local law, we refer to Sections 3.1 and 3.2 in \cite{LeeSchnelli2016}, where the local law for the Gaussian Orthogonal Ensemble (GOE) matrices is proved with the method. For a GOE matrix $W$, it is based on the observation
\begin{align*}
1+zG^W_{ii} = \sum_{k=1}^N W_{ik} G^W_{ki}
\end{align*}
and the expectation of the right-hand side can be obtained by the Stein lemma. Since
\begin{align*}
\mathbb{E}[W_{ik} G^W_{ki}] = -\mathbb{E}[G^W_{kk} G^W_{ii}] - \mathbb{E}[G^W_{ki} G^W_{ki}],
\end{align*}
the method suggests to consider the polynomial $1+zm^W+(m^W)^2$. In case the matrix is sparse, we need to add a correction term to the polynomial, which can be obtained from the following generalized Stein lemma, introduced in \cite{Stein1981} and applied to the proof of CLT for the linear eigenvalue statistics of random matrices by Lytova and Pastur \cite{LytovaPastur2009}.
\begin{lemma}\label{lemma:Stein}
 Fix $\ell\in\mathbb{N}$ and let $F\in C^{\ell +1} (\mathbb{R}; \mathbb{C}^{+})$. Let $Y$ be a centered random variable with finite moments to order $\ell +2$. Then,
\begin{equation}
\mathbb{E}[YF(Y)] = \sum_{r=1}^{\ell} \frac{\kappa ^{(r+1)}(Y)}{r!} \mathbb{E}[F^{(r)}(Y)] + \mathbb{E}[\Omega_{\ell}(YF(Y))],
\end{equation}
where $\mathbb{E}$ denotes the expectation with respect to $Y$, $\kappa^{(r+1)}(Y)$ denotes the $(r+1)$-st cumulant of~$Y$ and $F^{(r)}$ denotes the $r$-th derivative of the function~$F$. The error term $ \Omega_{\ell}(YF(Y))$ satisfies
\begin{align}
\mathbb{E}[\Omega_{\ell}(YF(Y))]& \leq C_{\ell} \mathbb{E}[|Y|^{\ell +2}] \sup_{|t|\leq Q} |F^{(\ell +1)}(t)| + C_{\ell}\mathbb{E}[|Y|^{\ell +2} \mathds{1}(|Y|>Q) \sup_{t\in \mathbb{R}}|F^{(\ell +1)}(t)|,
\end{align}
where $Q>0$ is an arbitrary fixed cutoff and $C_{\ell}$ satisfies $C_{\ell}\leq(C\ell)^\ell / \ell !$ for some numerical constant~$C$.
\end{lemma}

While the cumulant expansion method is applicable to the sample covariance matrices we consider, we face a new problem with the direct application of the method; since the Green function of $X^{\dag}X$ is $G^{X^{\dag}X}(z) = (X^{\dag}X-zI)^{-1}$, we are led to consider
\begin{equation*}
1+ zG^{X^{\dag}X}_{ii} = \sum_{k=1}^N (X^{\dag}X)_{ik} G^{X^{\dag}X}_{ki},
\end{equation*}
and the Stein lemma applied on the right-hand side proposes a cubic polynomial instead of the quadratic polynomial we use for the Wigner case. This makes the analysis significantly harder, especially when the correction term due to the sparsity is taken into consideration.

The difficulty is resolved by introducing the self-adjoint linearization of a sample covariance matrix $X^\dagger X$, which showed to be useful in the study of sample covariance matrices \emph{e.g.} in \cite{Girko1984,Anderson2013,LeeSchnelli2016,KnowlesYin2017,DingYang2017}. Let~$H$ be an $(N+M)\times (N+M)$ matrix such that 
\begin{align*}
H(X,z)\equiv H=THT+TH\overline{T}+\overline{T}HT+\overline{T}H\overline{T},
\end{align*}
where
\begin{align*}
THT:=-zI,\quad TH\overline{T}:=X^{\dagger}, \quad \overline{T}HT:=X, \quad \overline{T}H\overline{T}:=-I.
\end{align*}
Here $T$ is the projection on the first $N$ coordinates in $\mathbb{R}^{N+M}$ and $\overline{T}:=\mathds{1}-T$. In block matrix form, this is written as
\begin{align} \label{eq:linearization}
H= \left(
\begin{array} {c|c}
-zI & X^\dagger \\
\hline
X & -I
\end{array}
\right) .
\end{align}
Define the inverse matrix of $H(X, z)$ by $G:=H(X, z)^{-1}$. From the Schur complement formula, we find that
\begin{align}
G_{ab}(z)&=(TG(z)T)_{ab} = \big(X^{\dagger}X-zI\big)^{-1}_{ab}, \quad  1\leq a,b\leq N, \\
G_{\alpha\beta}(z)& = (\overline{T}G(z)\overline{T})_{\alpha\beta} = z\big(XX^{\dagger}-zI\big)^{-1}_{\alpha\beta},\quad N+1\leq \alpha,\beta\leq M+N.
\end{align}
Thus,
\begin{align}
\frac{1}{N}\sum^{N}_{a=1}G_{aa}= \frac{1}{N}\Tr \big(X^\dagger X-zI\big)^{-1} =m^{X^\dagger X}(z),
\end{align}
and
\begin{align}
\quad \frac{1}{N}\sum^{M+N}_{\alpha=N+1}G_{\alpha\alpha} &= \frac{1}{N}\Tr \Big(z\big(XX^{\dagger}-zI\big)^{-1}\Big) \notag \\
 &=zm^{X^\dagger X}(z)+\frac{1}{N}(N-M)=zm^{X^\dagger X}(z)+1-\frac{1}{d}.
\end{align}
We get last line from the fact that $XX^\dagger$ have $M$ eigenvalues identical to eigenvalues of $X^\dagger X$ and $(N-M)$ zero eigenvalues.

With the linearization trick, we can identify the quadratic polynomial that naturally arises from the sample covariance matrices, and we can also estimate the correction term as in \eqref{eq:polynomial}. With the quartic polynomial in \eqref{eq:polynomial}, we perform the analysis and prove the local law by the recursive moment method. The details for the recursive moment method are discussed in Appendix A of the Supplement~\cite{supplement}.

To prove the Tracy--Widom fluctuations of the largest eigenvalue, we consider the Dyson matrix flow with initial condition $X_0=X$ defined by
\begin{align}
X_t:=\mathrm{e}^{-t/2}X_0 +\sqrt{1-\mathrm{e}^{-t}}W^{G},\quad (t\geq 0),
\end{align}
where $W^{G}$ is an $M\times N$ matrix with i.i.d.\ Gaussian entries independent of $H_0$. The Dyson matrix flow is one of the key ideas in the proof of universality results in random matrix theory. For its application to the proof of Tracy--Widom limit, we refer to \cite{BourgadeErdosYau,LeeSchnelli2015,BourgadeHuangYau}.

In this work, we use the Dyson matrix flow in the Green function comparison in Section \ref{sec:TW} in conjunction with the linearization trick. Let $H_t :=H(X_t, z)$ be an $(N+M)\times (N+M)$ matrix defined as in \eqref{eq:linearization}. The local law can also be established for the normalized trace of the Green function of $H_t$ by considering
\begin{align}
G_t(z) = (H_t)^{-1}, \quad m_t(z) = \frac{1}{N} \sum_{i=1}^N (G_t)_{ii}(z),\quad (z\in\mathbb{C}^+).
\end{align}
Let $\kappa_t^{(k)}$ be the $k$-th cumulant of $(X_t)_{ij}$. Then, by the linearity of the cumulants under the addition of independent random variables, we have $\kappa_t^{(1)}=0, \kappa_t^{(2)}=1/N$ and $\kappa_t^{(k)}=\mathrm{e}^{-kt/2}\kappa^{(k)}$ for $k\geq 3$. In particular, we have the bound
\begin{align}
|\kappa_t^{(k)}| \leq \mathrm{e}^{-t}\frac{(Ck)^{ck}}{Nq_t^{k-2}} , \quad (k\geq 3),
\end{align}
where we introduced the time-dependent sparsity parameter
\begin{align}
q_t:= q\mathrm{e}^{t/2}.
\end{align}
We also let $s_t^{(k)}$ be the normalized cumulants defined by
\begin{equation}
s_t^{(k)} := N q_t^{k-2} \kappa_t^{(k)}.
\end{equation}

With the parameters depending on $t$, we generalize Theorem \ref{thm:norm} as follows: 

\begin{proposition} \label{prop:locallaw}
Let $X_0$ satisfy Assumption \ref{assume} with $\phi>0$. Then, for any $t\geq 0$, 
there exist deterministic numbers $L_t^+ \geq L_t^- \geq 0$ and an algebraic function $\widetilde{m}_t:\mathbb{C}^+ \rightarrow \mathbb{C}^+$ 
such that the following hold:
\begin{enumerate}
\item[(1)] The function $\widetilde{m}_t$ is the Stieltjes transform of a deterministic probability measure $\widetilde{\rho}_t$, i.e., $\widetilde{m}_t(z) = m_{\widetilde{\rho}_t}(z).$ The measure $\rho_t$ is supported on $[L_t^-, L_t^+]$ and $\widetilde{\rho}_t$ is absolutely continuous with respect to Lebesgue measure with a strictly positive density on $(L_t^-, L_t^+)$.
\item[(2)] The function $\widetilde{m}_t\equiv\widetilde{m}_t(z)$, $z\in\mathbb{C}^+$, is a solution to the polynomial equation
\begin{align}
P_z(\widetilde{m}_t):=1+(z+1-\frac{1}{d})\widetilde{m}_t+z\widetilde{m}_t^2+\frac{s_t^{(4)}}{q_t^2}\widetilde{m}_t^2(z\widetilde{m}_t+1-\frac{1}{d})^2 = 0.
\end{align}
\item[(3)] The normalized trace $m_t^{X^\dagger X}$ of the Green function of $X_t^\dagger X_t$ satisfies the local law
\begin{align}\label{eq:locallaw_t}
|m_t^{X^\dagger X}(z) -\widetilde{m}_t(z)| \prec \frac{1}{q_t^2}+\frac{1}{N\eta},
\end{align}
uniformly on the domain $\mathcal{E}$ and uniformly in $t\in [0, 6\log N]$.
\end{enumerate}
\end{proposition}
We remark that the local eigenvalue statistics of $X_t$ and $W^G$ agree up to negligible error for $t \geq 6\log N$. 

For simplicity, we let $L_t \equiv L_t^+$, the upper edge of the support of $\widetilde{\rho}_t$. In Section \ref{sec:stieltjes}, we show that 
\begin{align} \label{eq:L_t}
L_t =\Big(1+\frac{1}{\sqrt d}\Big)^2+\frac{1}{\sqrt d}\big(1+\frac{1}{\sqrt d}\big)^2\mathrm{e}^{-t}q_t^{-2}s^{(4)}+O(\mathrm{e}^{-2t}q_t^{-4}),
\end{align}
and also satisfies
\begin{align} \label{eq:L'_t}
\dot{L}_t = -\frac{2}{\sqrt d}\big(1+\frac{1}{\sqrt d}\big)^2\mathrm{e}^{-t}q_t^{-2}s^{(4)}+O(\mathrm{e}^{-2t}q_t^{-4}),
\end{align}
where $\dot{L}_t$ denotes the derivative of $L_t$ with respect to $t$. The actual proof of the Tracy--Widom fluctuation in Section \ref{sec:TW} will be done by comparing $\dot{L}_t$ and the time change of a suitable functional of the Green function.

In the proof of the local law, Theorem \ref{thm:locallaw}, we use the following results of \cite{DingYang2017} as a priori estimates.

\begin{proposition} (Lemma 3.11 of \cite{DingYang2017}) \label{prop26} Suppose $X$ satisfies Assumption \ref{assume} with $\phi>0$. Then,
\begin{enumerate}
\item[(1)] (Local Marchenko Pastur law) The following estimates hold uniformly for $z\in\mathcal{E}$:
\begin{equation}\label{eq:localMP1}
|m^{X^\dagger X}(z)-m_{\mathrm{MP}}(z)|\prec\min \Big\{\frac{1}{q}, \frac{1}{q^2 \sqrt{\kappa+\eta}}\Big\} + \frac{1}{N\eta},
\end{equation}
\begin{equation}\label{eq:localMP2}
\max_{\mathfrak{i},\mathfrak{j}}|G_{\mathfrak{i}\mathfrak{j}}(z)-\delta_{\mathfrak{i}\mathfrak{j}}\Pi_{\mathfrak{i}\mathfrak{j}}(z))|\prec \frac{1}{q}+\sqrt{\frac{\Im m_{\mathrm{MP}}(z)}{N\eta}} +\frac{1}{N\eta},
\end{equation}
where $\kappa\equiv\kappa(z):=|E-\lambda_+|,$ $z=E+\ii\eta$, and
\begin{align*}
\Pi(z)= \left(
\begin{array} {c|c}
m_{\mathrm{MP}}(z)I_{N\times N} & 0 \\
\hline
0 & -(1+m_{\mathrm{MP}}(z))^{-1}I_{M\times M}
\end{array}
\right) .
\end{align*}
\item[(2)] (Bound on $\|H \|$) There exists a constant $\Lambda>0$ such that
\begin{equation}
\|H \|\leq \Lambda .
\end{equation}
\item[(3)] (Delocalization) For the $\ell^2$-normalized eigenvectors $(\mathbf{u}_k^{X^\dagger X})$ and $(\mathbf{v}_{\alpha}^{XX^{\dagger}})$,
\begin{equation}\label{eq:deloc}
\max_{k}\|\mathbf{u}_k^{X^\dagger X}\|_{\infty} + \max_{\alpha} \|\mathbf{v}_{\alpha}^{XX^{\dagger}} \|_{\infty} \prec\frac{1}{\sqrt{N}}.
\end{equation}
\item[(4)] (Rigidity) Let $\gamma_j$ be the classical location of the $j$-th eigenvalue of $X^{\dagger}X$, i.e., $\gamma_j$ is defined by
\begin{align*}
N\int_{-\infty}^{\gamma_j}\rho_{\mathrm{MP}}(x)\di x=j, \quad 1\leq j\leq N.
\end{align*} If $\phi>1/3$, then
\begin{align} \label{rigidity}
|\lambda_j-\gamma_j|\prec j^{-1/3}N^{-2/3}+q^{-2},
\end{align}
for $\gamma_j\in[L_+-c, L_+]$ where $c$ is sufficiently small constant.
\end{enumerate}
\end{proposition}
Note that the estimate \eqref{eq:localMP1} is essentially optimal as long as the spectral parameter $z$ stay away from the spectral edges, \emph{e.g.} for energies in the bulk $E\in[\lambda_- +\delta, \lambda_+-\delta]$ for some ($N$-independent) $\delta>0$. For the individual Green function entries $G_{i\alpha}$, we believe that the estimate \eqref{eq:localMP2} is already essentially optimal.

\begin{remark}[Notational remark 2]
We use Latin letters for indices in $[1, N]$, Greek letters for indices in $[N+1, M+N]$, and fraktur letters for indices ranging from $1$ to $M+N$. For Latin, respectively Greek,  indices, we abbreviate
\begin{align*}
\sum_{i} := \sum_{i=1}^N\;,\qquad \sum_{\alpha} := \sum_{\alpha=N+1}^{M+N}.
\end{align*}
For simplicity, we also let
\begin{align*}
\sum_{i,\alpha} := \sum_{\substack{1\leq i \leq N \\ N+1\leq \alpha \leq N+M}}\;,\qquad\sum_{\mathfrak{i}}:=\sum_{\mathfrak{i}=1}^{M+N}.
\end{align*}
\end{remark}

\section{Stieltjes transform of $\widetilde{\rho}$} \label{sec:stieltjes}

In this section, we prove several properties of $\widetilde{m}_t$ and its Stieltjes inversion $\widetilde\rho_t$.

\begin{lemma}\label{lemma:rho_t} For fixed $z=E+\ii\eta\in\mathcal{E}$ and $t\geq 0$, the equation $P_{t,z}(w_t)=0$ has a unique solution $w_t\equiv w_t(z)$ satisfying $\Im w_t>0$ and $|w_t|\leq \frac{6\lambda_+}{\lambda_-}$. Furthermore, $w_t$ satisfies the following properties:
\begin{enumerate}
\item[(1)] There exists a probability measure $\widetilde{\rho}_t$ such that the analytic continuation of $w_t(z)$ coincides with the Stieltjes transform, $m_{\widetilde{\rho}_t}(z)$, of $\widetilde{\rho}_t$.
\item[(2)] The probability measure $\widetilde{\rho}_t$ is supported on $[L_t^-, L_t^+]$ for some $L_t^- \geq 0$ and $L_t^+ \equiv L_t\geq (1+\sqrt{1/d})^2$, and it exhibits a square-root decay at the upper edge, $i.e.$
\begin{align}
\widetilde{\rho}_t(E)\sim\sqrt{L_t -E},\quad (E\in [L_t -\frac{1}{\sqrt d},L_t]).
\end{align}
Moreover,
\begin{align}
L_t = \Big(1+\frac{1}{\sqrt d}\Big)^2+\frac{1}{\sqrt d}\big(1+\frac{1}{\sqrt d}\big)^2\mathrm{e}^{-t}q_t^{-2}s^{(4)}+O(\mathrm{e}^{-2t}q_t^{-4}).
\end{align}
\item[(3)] Setting
\begin{align}
\varkappa_t\equiv\varkappa_t(E) :=\min\{|E+L_t|, |E-L_t|\},
\end{align}
the solution $w_t$ satisfies that
\begin{align}
|P_{t,z}'(w_t)| \sim \sqrt{\varkappa_t(E)+\eta}
\end{align}

and
\begin{align}
\Im w_t(E+\ii\eta)\sim\frac{\eta}{\sqrt{\varkappa_t(E)+\eta}} &\qquad\textrm{if }E > L_t \,, \notag \\
\Im w_t(E+\ii\eta)\sim\sqrt{\varkappa_t(E)+\eta} &\qquad\textrm{if } L_t -\frac{1}{\sqrt d} \leq E \leq L_t \,, \\
\Im w_t(E+\ii\eta) =O(1) &\qquad\textrm{if }E < L_t -\frac{1}{\sqrt d} \,. \notag 
\end{align}
\end{enumerate}
\end{lemma}

\begin{proof}
Assume first that $d > 1$. Recall the definition of $P \equiv P_{t,z}$,
\begin{align}
P(w) =1+(z+1-\frac{1}{d})w+zw^{2}+\frac{s^{(4)}\mathrm{e}^{-t}}{q_t^2}w^2(zw+1-\frac{1}{d})^2.
\end{align}
Solving the equation $P=0$ for $z$ by the quadratic formula, we get
\begin{align}
z &= Q(w)\equiv Q_{z,t}(w_t) \\
&= \frac{-\Big( 2\mathrm{e}^{-t}q_t^{-2}s^{(4)}(1-\frac{1}{d})w^2 + w + 1\Big) +\sqrt{-4\mathrm{e}^{-t}q_t^{-2}s^{(4)}w^2 \cdot\frac{1}{d}+(w+1)^2}}{2\mathrm{e}^{-t}q_t^{-2}s^{(4)}w^3}. \notag
\end{align}
For $w \sim 1$, its derivative is given by
\begin{align}
Q'(w) &= \frac{1}{2\mathrm{e}^{-t}q_t^{-2}s^{(4)}} \left( 2\mathrm{e}^{-t}q_t^{-2}s^{(4)}(1-\frac{1}{d}) w^{-2} + 2w^{-3} +3w^{-4} \right. \notag \\
&\qquad \qquad \qquad \quad \left. + \frac{8\mathrm{e}^{-t}q_t^{-2}s^{(4)} w^{-5} \cdot \frac{1}{d} - (w^{-3}+w^{-2})(2w^{-3}+3w^{-4})}{\sqrt{-4\mathrm{e}^{-t}q_t^{-2}s^{(4)} w^{-4} \cdot \frac{1}{d} + (w^{-3}+w^{-2})^2}} \right) \notag \\
&= \frac{1}{2\mathrm{e}^{-t}q_t^{-2}s^{(4)}} \bigg( 2w^{-3}+3w^{-4} 
-(2w^{-3}+3w^{-4}) \left[ 1- 4\mathrm{e}^{-t}q_t^{-2}s^{(4)} \frac{1}{d} \frac{w^2}{(1+w)^2} \right]^{-1/2} \\
&\qquad \qquad \qquad \quad + 2\mathrm{e}^{-t}q_t^{-2}s^{(4)}(1-\frac{1}{d}) w^{-2} + 8\mathrm{e}^{-t}q_t^{-2}s^{(4)} \frac{1}{w^2(1+w)} \cdot \frac{1}{d} \bigg) + O(q_t^{-2}) \notag \\
&= -\frac{2w+3}{w^2(1+w)^2} \cdot \frac{1}{d} + (1-\frac{1}{d}) w^{-2} + \frac{4}{w^2(1+w)} \cdot \frac{1}{d} + O(\mathrm{e}^{-t}q_t^{-2}). \notag
\end{align}
Let
\begin{align}
\widetilde{Q}'(w) = -\frac{2w+3}{(1+w)^2} \cdot \frac{1}{d} + (1-\frac{1}{d}) w^{-2} + \frac{2}{1+w} = \frac{1}{w^2} - \frac{1}{d} \frac{1}{(1+w)^2}.
\end{align}
Then, $\widetilde{Q}'(w)$ is independent of $q_t$ and strictly increasing on $(-1, 0)$. Furthermore,
\begin{align}
\widetilde{Q}'\left(-\frac{1}{1+\frac{1}{\sqrt d}} \right) = 0.
\end{align}
Thus, there exists a unique solution $w=\tau_t$ of the equation $Q'(w)=0$ in $(-1, 0)$, which satisfies
\begin{align}
\tau_t = -\frac{1}{1+\frac{1}{\sqrt d}} + O(\mathrm{e}^{-t}q_t^{-2}).
\end{align}
Let $L_t = Q(\tau_t)$. Calculating the terms of order $\mathrm{e}^{-t}q_t^{-2}$ precisely, we get  
\begin{align}
\tau_t &= \frac{1}{-1-\frac{1}{\sqrt d}}+O(\mathrm{e}^{-2t}q_t^{-4}), \\
 L_t &=\Big(1+\frac{1}{\sqrt d}\Big)^2+\frac{1}{\sqrt d}\big(1+\frac{1}{\sqrt d}\big)^2\mathrm{e}^{-t}q_t^{-2}s^{(4)}+O(\mathrm{e}^{-2t}q_t^{-4}). \notag
\end{align}

For simplicity, we let $L \equiv L_t$ and $\tau = \tau_t$. We now expand $z$ about $\tau$ to find that
\begin{align} \label{eq:sq-branch0}
z &= Q(\tau) + Q'(\tau)(w-\tau) +\frac{Q''(\tau)}{2}(w-\tau)^2 + O(|w-\tau|^3) \notag \\
&= L+ \frac{Q''(\tau)}{2}(w-\tau)^2 + O(|w-\tau|^3)
\end{align}
in a $q_t^{-1/2}$-neighborhood of $\tau$. Since 
\begin{align}
Q''(\tau) = \widetilde{Q}''(\tau) + O(\mathrm{e}^{-t}q_t^{-2})
\end{align}
and $\widetilde{Q}''(\tau) \sim 1$ where $\widetilde{Q}'$ is monotone increasing on $(-1, 0)$, we find that $Q''(\tau) > 0$. We hence find that
\begin{align} \label{eq:sq-branch}
w = \tau + \left(\frac{2}{Q''(\tau)}\right)^{1/2} \sqrt{z-L} + O(|z-L|)
\end{align}
in this neighborhood. Choosing the branch of the square root so that $\sqrt{z-L} \in \mathbb{C}^+$, we find that $\Im w > 0$. 

Let $B_0 := \{ w \in \mathbb{C} : |w| < \frac{6\lambda_+}{\lambda_-} \}$. For $z \in \mathcal{E}$ and $|w| = \frac{6\lambda_+}{\lambda_-}$,
\begin{align}
|zw^2| = \frac{|w|}{2}|zw| + \frac{|zw|}{2}|w| \geq 3|zw| + \lambda_+ |w| > 2 + |zw| + \left| 1-\frac{1}{d} \right| |w|
\end{align}
hence
\begin{align}
&\left| 1+(z+1-\frac{1}{d})w+zw^{2} \right| \geq |zw^2| - \left| (z+1-\frac{1}{d})w \right| -1 \\
&> 1 > \left| \frac{s^{(4)}\mathrm{e}^{-t}}{q_t^2}w^2(zw+1-\frac{1}{d})^2 \right|. \notag
\end{align}
Now, Rouch\'e's theorem implies that the polynomial $P(w)$ has the same number of roots as the quadratic polynomial $1+(z+1-\frac{1}{d})w+zw^{2} =0$ in $B_0$. Hence, we conclude that $P(w)=0$ has two solutions on $B_0$.

Let us extend $w \equiv w(z)$ to cover $z \in \partial \mathcal{E} \cap \mathbb{R}$. Then, $w$ forms a curve $w:\partial \mathcal{E} \cap \mathbb{R} \to \mathbb{C}$, which we will denote by $\Gamma$. We already know that $\Gamma$ intersects the real axis at $\tau$. Let $\widetilde{\tau}$ be the largest real number such that $\widetilde{\tau} < \tau$ and $\Gamma$ intersects the real axis at $\widetilde{\tau}$. Since
\begin{align}
1+(\widetilde{\tau}+1-\frac{1}{d})w(\widetilde{\tau})+\widetilde{\tau} w(\widetilde{\tau})^2 = O(\mathrm{e}^{-t}q_t^{-2}),
\end{align}
it can be easily checked from the quadratic formula 
\begin{align} \label{eq:quadratic_w}
w&= \frac{1}{2z} \left[ -(z+1-\frac{1}{d}) \right.  \left. + \sqrt{(z+1-\frac{1}{d})^2 -4z \left( 1+ \mathrm{e}^{-t}q_t^{-2}s^{(4)} w^2(zw+1-\frac{1}{d})^2 \right)} \right] \notag
\end{align}
that $m_{\mathrm{MP}}(\widetilde{\tau}) - w(\widetilde{\tau}) = O(\mathrm{e}^{-t}q_t^{-2})$ and also $|\widetilde{\tau} - \lambda_-| = O(\mathrm{e}^{-t/2}q_t^{-1})$, where $m_{\mathrm{MP}}$ is the Stieltjes transform of the  Marchenko--Pastur law; see~\eqref{le MP}. Since we chose the branch of the square root in \eqref{eq:sq-branch} so that $\sqrt{z-L} \in \mathbb{C}^+$, we find that the curve $\Gamma \in \mathbb{C}^+ \cup \mathbb{R}$, joining $\tau$ and $\widetilde{\tau}$. This shows that one solution of $P(w)=0$ is in $\mathbb{C}^+$. Choosing the other branch for the square root in \eqref{eq:sq-branch}, we can identify another solution of $P(w)=0$ in $\mathbb{C}^-$. Since there are only two solutions of $P(w)=0$ in $B_0$, this proves the uniqueness statement of the lemma.
Furthermore, by the analytic inverse function theorem, we also find that $w(z)$ is analytic for $z \in (w^{-1}(\widetilde{\tau}), L)$ since $Q'(w) \neq 0$ for such~$z$.

For a large but $N$-independent $z$, we can find from the quadratic formula \eqref{eq:quadratic_w} that $w(z) = -\frac{1}{z} + o(\frac{1}{z})$. By continuity, this shows that the analytic continuation of $w(z)$ for $z \in \mathbb{C}^+$ is in the domain $D_{\Gamma}$ enclosed by $\Gamma$ and the real axis. In particular, $|w(z)| < \frac{6\lambda_+}{\lambda_-}$ for all $z \in \mathbb{C}^+$.

We then prove the analyticity of $w(z)$ in $\mathbb{C}^+$. It suffices to show that $Q'(w) \neq 0$ for $w \in D_{\Gamma}$. If $Q'(w) = 0$ for some $w \in D_{\Gamma}$, we have
\begin{align}
0 = w^2 Q'(w) = 1 - \frac{1}{d} \frac{w^2}{(1+w)^2} + O(\mathrm{e}^{-t}q_t^{-2}).
\end{align}
We again use Rouch\'e's theorem. Since $d \geq 1$, we have $1 - \frac{1}{d} \frac{w^2}{(1+w)^2} > c \gg \mathrm{e}^{-t}q_t^{-2}$ for $|w|=\frac{6\lambda_+}{\lambda_-}$. Hence, $w^2 Q'(w) = 0$ has two solutions in the disk $B_0$. We already know that those solutions are $\tau$ and $\widetilde{\tau}$. Thus, $Q'(w) \neq 0$ for $w \in D_{\Gamma}$ and $w(z)$ is analytic.

Let $\widetilde{\rho} \equiv \widetilde{\rho}$ be the Stieltjes inversion of $w \equiv w(z)$. In order to show that $\widetilde{\rho}$ is a probability measure, it suffices to show that $\lim_{y \to \infty} \ii y \, w(\ii y) = -1$. By considering $z = \ii y$ in \eqref{eq:quadratic_w}, it can be easily checked. This proves the first part of the lemma.

The second part of the lemma is already proved in the previous computation in the proof. To prove the last part of the lemma, with a slight abuse of notation, let $P_t (w, z) = P_{t,z}(w)$. We notice that
\begin{align}
0 = \frac{\di}{\di w} P_t (w, z) = \frac{\partial z}{\partial w} \cdot \frac{\partial}{\partial z} P(w, z) + \frac{\partial}{\partial w} P(w, z).
\end{align}
From \eqref{eq:sq-branch0} and \eqref{eq:sq-branch}, we find that
\begin{align}
\frac{\partial z}{\partial w} \sim \sqrt{z-L} \sim \sqrt{\varkappa+\eta}.
\end{align}
We claim that $\frac{\partial}{\partial z} P(w, z) \sim 1$, which would prove the first relation in the last part of the lemma. Since
\begin{align}
\frac{\partial}{\partial z} P(w, z) = w+w^2 + \frac{2s^{(4)}\mathrm{e}^{-t} w}{q_t^2}w^2(zw+1-\frac{1}{d}),
\end{align}
it suffices to prove that 
\begin{align} \label{eq:stability}
|w|, |1+w|>c>0,
\end{align}
for some constant $c$ independent of $N$. If we assume $|w|\leq c$, then
\begin{align}
|P(w, z)| \geq 1- c \left| z+1 -\frac{1}{d} \right| - c^2 |z|-C q_t^{-2} > c,
\end{align}
for some (small) $c>0$, which contradicts that $P(w, z)=0$. Similarly, if we assume $|1+w| \leq c$, then
\begin{align}
|P(w, z)| \geq \frac{1}{d} - |1+w| - |zw| \cdot |1+w| - \frac{|1+w|}{d} - C q_t^{-2} > c,
\end{align}
for some (small) $c>0$. This proves that $\frac{\partial}{\partial z} P(w, z) \sim 1$, and we find that 
\begin{align}
P_{t, z}'(w) = \frac{\partial}{\partial w} P(w, z) =- \frac{\partial z}{\partial w} \cdot \frac{\partial}{\partial z} P(w, z) \sim \sqrt{\varkappa+\eta}.
\end{align} 
This proves the first relation in the last part of the lemma. Other relations in the last part of the lemma can be easily proved from the first property and \eqref{eq:sq-branch}.

If $d=1$, the polynomial $P(w)$ reduces to
\begin{align}
P(w) = 1+zw+zw^{2}+\mathrm{e}^{-t}q_t^{-2}s^{(4)} z^2 w^4
\end{align}
and
\begin{align}
Q(w) = \frac{-(w + 1) +\sqrt{-4\mathrm{e}^{-t}q_t^{-2}s^{(4)}w^2 +(w+1)^2}}{2\mathrm{e}^{-t}q_t^{-2}s^{(4)}w^3}.
\end{align}
The square root behavior does not change in this case, with $\tau_t = -\frac{1}{2} + O(\mathrm{e}^{-t}q_t^{-2})$. For uniqueness, we consider the disk $B_{10} = \{ w \in \mathbb{C}: |w| < 10 \}$. For $z \in \mathcal{E}$ and $w \in \partial B_{10}$, $|zw^2| \geq |zw| + 2$. Hence, we can use Rouch\'e's theorem and the uniqueness statement follows.

For the analyticity, if $Q'(w)=0$ then
\begin{align}
0 = w^2 Q'(w) = 1 - \frac{w^2}{(1+w)^2} + O(\mathrm{e}^{-t}q_t^{-2}).
\end{align}
In this case, the equation $1 - \frac{w^2}{(1+w)^2} = 0$ has only one solution $w = -\frac{1}{2}$. Thus, $w^2 Q'(w) = 0$ has only one solution in the disk $B_{10}$, and it proves the analyticity. The remaining parts can be proved with suitable changes.
\end{proof}

\section{Proof of local laws} \label{sec:local}

\subsection{Proof of Proposition \ref{prop:locallaw}} 
In this subsection, we prove Proposition \ref{prop:locallaw}. The following lemma provides the main tool of the proof, which is the recursive moment estimate for $P(m_t)$. It is analogous to Lemma~5.1 in~\cite{LeeSchnelli2016} for sparse Wigner matrices.
\begin{lemma}{(Recursive moment estimate)} \label{lemma:recursive} Fix $\phi>0$ and $t\geq 0$. Let $X_0$ satisfies Assumption \ref{assume}. Then, for any $D>10$ and small $\epsilon>0$, the normalized trace of the Green function, $m_t\equiv m_t(z)$, of the matrix $H_t$ satisfies
\begin{align}\label{eq:recursive}
&\mathbb{E}|P(m_t)|^{2D}\leq N^{\epsilon}\mathbb{E}\Big[\Big(\frac{1}{q_t^4} +\frac{\Im m_t}{N\eta} + \frac{N-M}{N^2}\Big)|P(m_t)|^{2D-1}\Big] \\
&\qquad + N^{-\epsilon/4}q_t^{-1}\mathbb{E}\Big[|m_t -\widetilde{m}_t|^2|P(m_t)|^{2D-1}\Big] +N^{\epsilon}q_t^{-8D}  \notag \\
&\qquad +N^{\epsilon} q_t^{-1}\sum_{s=2}^{2D}\sum_{u'=0}^{s-2}\mathbb{E}\Big[\Big(\frac{\Im m_t}{N\eta}+\frac{N-M}{N^2}\Big)^{2s-u'-2}|P'(m_t)|^{u'}|P(m_t)|^{2D-s}\Big] \notag \\
%&\qquad \notag \\
&\qquad +N^{\epsilon}\sum_{s=2}^{2D}\mathbb{E}\Big[\Big(\frac{1}{N\eta}+\frac{1}{q_t}\Big(\frac{\Im m_t}{N\eta} + \frac{N-M}{N^2}\Big)^{1/2}+\frac{1}{q_t^2}\Big) \Big(\frac{\Im m_t}{N\eta} + \frac{N-M}{N^2}\Big)^{s-1}|P'(m_t)|^{s-1}|P(m_t)|^{2D-s}\Big], \notag
\end{align}
uniformly on the domain $\mathcal{E}$, for any sufficiently large $N$.
\end{lemma}

We prove Lemma \ref{lemma:recursive} in Appendix A of the Supplement~\cite{supplement}. In the following, we sketch the proof of the local law in~Proposition~\ref{prop:locallaw}; the details are found in Appendix B of the Supplement~\cite{supplement}. The remaining parts of this section are partly adapted from Section 4 in~\cite{LeeSchnelli2016}, and we reproduce the argument here in a more structured and clear way.

 Let $\widetilde{m}_t$ be the solution $w_t$ in Lemma \ref{lemma:rho_t}. To simplify the notation, we introduce the following $z$- and $t$-dependent deterministic parameters
\begin{align}\label{le alphas}
\alpha_1(z) := \Im \widetilde{m}_t(z), \quad \alpha_2(z):=P'(\widetilde{m}_t(z)), \quad \beta:= \frac{1}{N\eta} + \frac{1}{q_t^2},
\end{align}
with $z=E+\ii\eta$. 
From Lemma \ref{lemma:rho_t}, we check that $\alpha_1\le C |\alpha_2|$. Further let 
\begin{align}\label{le Lambda_t}
\Lambda_t(z):=|m_t(z) - \widetilde{m}_t(z) |,\quad (z\in\mathbb{C}^+).
\end{align}
Note that from~Proposition~\ref{prop26} and~\eqref{eq:quadratic_w}, we have that $\Lambda_t(z)\prec 1$ uniformly on $\mathcal{E}$.

The strategy is now as follows. We apply Young's inequality to split up all the terms on the right side of~\eqref{eq:recursive} and absorb resulting factors of $\mathbb{E}|P(m_t)|^{2D}$ into the left hand side. For example, for the first term on the right of~\eqref{eq:recursive}, we get, upon using the notation in~\eqref{le Lambda_t}, that
\begin{align}
N^{\epsilon}\Big(&\frac{\Im m_t}{N\eta} +\frac{N-M}{N^2}+ q_t^{-4}\Big)|P(m_t)|^{2D-1}\\ &\leq  N^{\epsilon} \frac{\alpha_1 + \Lambda_t}{N\eta}|P(m_t)|^{2D-1}+ N^{\epsilon}  q_t^{-4}|P(m_t)|^{2D-1} \notag\\
&\leq \frac{N^{(2D+1)\epsilon}}{2D}C^{2D}\beta^{2D}(\alpha_1+\Lambda_t )^{2D} + \frac{N^{(2D+1)\epsilon}}{2D}q_t^{-8D}+\frac{2(2D-1)}{2D}N^{-\frac{\epsilon}{2D-1}}|P(m_t)|^{2D}, \notag 
\end{align}
and note that the last term can be absorbed into the left side of~\eqref{eq:recursive}. The same idea can be applied to the second term on the right side of~\eqref{eq:recursive}. To hand the other terms, we Taylor expand $P'(m_t)$ around $\widetilde{m}_t$ as
\begin{align}
|P'(m_t) - \alpha_2 -P''(\widetilde{m}_t)(m_t-\widetilde{m}_t)| \leq Cq_t^{-2}\Lambda_t^2,
\end{align}
where we used~\eqref{le alphas}. Noticing that $P''(\widetilde{m}_t)=2z+O(q_t^{-2})$ , we proceed in a similar way as above using Young's inequality. Skipping over some details, we eventually find
\begin{align}\label{eq:recursive3bis}
&\mathbb{E}[|P(m_t)|^{2D}]\\
&\leq CN^{(2D+1)\epsilon}\mathbb{E}[\beta^{2D}(\alpha_1 +\Lambda_t)^D(|\alpha_2|+15\Lambda_t)^D]+C\frac{N^{(2D+1)\epsilon}}{2D}q_t^{-8D} +C\frac{N^{-(D/4-1)\epsilon}}{2D}q_t^{-2D}\mathbb{E}[\Lambda_t^{4D}] \notag \\
&\leq N^{3D\epsilon}\beta^{2D}|\alpha_2|^{2D} +N^{3D\epsilon}\beta^{2D}\mathbb{E}[\Lambda_t^{2D}]+N^{3D\epsilon}q_t^{-8D}+N^{-D\epsilon/8}q_t^{-2D}\mathbb{E}[\Lambda_t^{4D}], \notag
\end{align}
uniformly on $\mathcal{E}$, where we used $\alpha_1\le C\alpha_2$ to get the second line. 

Next, we aim to control $\Lambda_t$ in terms of $|P(m_t)|$. For that we Taylor expand $P(m_t)$ around $\widetilde m_t$ to get
\begin{align} \label{eq:Taylorbis}
\Big|P(m_t)-\alpha_2(m_t-\widetilde{m}_t)-\frac{1}{2}P''(\widetilde{m}_t)(m_t-\widetilde{m}_t)^2\Big|\leq Cq_t^{-2}\Lambda_t^3,
\end{align}
since $P(\widetilde{m}_t)=0$ and $P'''(\widetilde{m}_t)=8\mathrm{e}^{-t}q_t^{-2}s^{(4)}(z^2 \widetilde{m}_t+z(z\widetilde{m}_t +1-\frac{1}{d}))$. Then using $\Lambda_t\prec 1$ and $P''(\widetilde{m}_t)=2z+O(q_t^{-2})$ we obtain
\begin{align}
\Lambda_t^2\prec 2|\alpha_2|\Lambda_t +2|P(m_t)|, \quad (z\in\mathcal{E}).
\end{align}
This estimate can upon applying a Schwarz inequality be fed back into~\eqref{eq:recursive3bis}, to get the bound
\begin{align}
 \mathbb{E}[|P(m_t)|^{2D}]&\leq N^{5D\epsilon}\beta^{2D}|\alpha_2|^{2D}+N^{5D\epsilon}\beta^{4D}+q_t^{-2D}|\alpha_2|^{4D},
\end{align}
 uniformly on $\mathcal{E}$. For any fixed $z\in\mathcal{E}$, Markov's inequality then yields $|P(m_t)|\prec|\alpha_2|\beta+\beta^2+q_t^{-1}|\alpha_2|^2$. Uniformity in $z$ is easily achieved using a lattice argument and the Lipschitz continuity of $m_t(z)$ and $\widetilde{m}_t(z)$ on $\mathcal{E}$. Finally, a Taylor expansion of $P(m_t)$ around $\widetilde{m}_t$ will give the following self-consistent equation for $m_t(z)-\widetilde{m}_t(z)$:
\begin{lemma}\label{le new lemma 1}
Under the assumptions of Lemma~\ref{lemma:recursive}, we have
\begin{align} \label{eq:TaylorPbis}
 |\alpha_2(m_t-\widetilde{m}_t)+z(m_t-\widetilde{m}_t)^2|\prec\beta\Lambda_t^2 +|\alpha_2|\beta+\beta^2+q_t^{-1}|\alpha_2|^2,
\end{align}
uniformly on $\mathcal{E}$.
\end{lemma}
The detailed proof of Lemma~\ref{le new lemma 1} is given in Appendix B of the Supplement~\cite{supplement}.

We remark that the last term on the right-hand side of~\eqref{eq:TaylorPbis} is not yet optimal. To get a better estimate, we use the behavior of the $\alpha_2(z)$ to refine the analysis and obtain the following corollary.
\begin{corollary}
Under the assumptions of Lemma~\ref{lemma:recursive}, we have
\begin{align} \label{eq:TaylorPtre}
 |\alpha_2(m_t-\widetilde{m}_t)+z(m_t-\widetilde{m}_t)^2|\prec\beta\Lambda_t^2 +|\alpha_2|\beta+\beta^2,
\end{align}
uniformly on $\mathcal{E}$.

\end{corollary}
\begin{proof}
Recall from Lemma~\ref{lemma:rho_t} that there is a constant $C_0>1$ such that $C_0^{-1}\sqrt{\varkappa_t(E)+\eta}\leq|\alpha_2|\leq C_0\sqrt{\varkappa_t(E)+\eta}$, where we can choose $C_0$ uniform in $z\in\mathcal{E}$. Note that, for a fixed $E$, $\beta=\beta(E+\ii\eta)$ is a decreasing function of $\eta$ whereas $\sqrt{\varkappa_t(E)+\eta}$ is increasing. Hence there is $\widetilde{\eta_0}\equiv\widetilde{\eta_0}(E)$ such that $\sqrt{\varkappa(E)+\widetilde{\eta_0}}=C_0q_t\beta(E+\ii\widetilde{\eta}_0)$. We consider the subdomain $\widetilde{\mathcal{E}}\subset\mathcal{E}$ defined by
\begin{align}
\widetilde{\mathcal{E}}:=\{z=E+\ii\eta\in\mathcal{E}:\eta>\widetilde{\eta}_0(E)\}.
\end{align}
On this subdomain $\widetilde{\mathcal{E}}$, $\beta\leq q_t^{-1}|\alpha_2|$, hence we get from \eqref{eq:TaylorPbis} that there is a high probability event $\widetilde{\Upxi}$ such that
\begin{align*}
|\alpha_2(m_t-\widetilde{m}_t)+z(m_t-\widetilde{m}_t)^2|\leq N^{\epsilon}\beta\Lambda_t^2+N^\epsilon q_t^{-1}|\alpha_2|^2
\end{align*}
and thus
\begin{align*}
|\alpha_2|\Lambda_t\leq (|z|+N^\epsilon \beta)\Lambda_t^2 +N^\epsilon q_t^{-1}|\alpha_2|^2
\end{align*}
uniformly on $\widetilde{\mathcal{E}}$ on $\widetilde{\Upxi}$. Hence, on $\widetilde{\Upxi}$,
\begin{align}
|\alpha_2|\leq 2(|z|+1)\Lambda_t\leq 12\Lambda_t\quad \textrm{or}\quad \Lambda_t\leq 2N^\epsilon q_t^{-1}|\alpha_2|,\quad(z\in\widetilde{\mathcal{E}}).
\end{align}
When $\eta=N^{-\epsilon}$, it is easy to see that
\begin{align}
|\alpha_2|\geq|z+1-\frac{1}{d}+2z\widetilde{m}_t|-Cq_t^{-2}\geq 2E\Im\widetilde{m}_t \geq c\sqrt{\eta}\gg 2N^\epsilon q_t^{-1}|\alpha_2|,
\end{align}
for some constant $c$ and sufficiently large $N$. We have that either $N^{-\epsilon/2}/12\leq \Lambda_t$ or $\Lambda_t\leq 2N^\epsilon q_t^{-1}|\alpha_2|$ on $\widetilde{\Upxi}$. From the a priori estimate \eqref{eq:localMP1}, we know that $|\Lambda_t|\prec \frac{1}{q_t}+\frac{1}{N\eta}$, we hence find that
\begin{align}\label{le dicho}
\Lambda_t\leq 2N^\epsilon q_t^{-1}|\alpha_2|,
\end{align}
holds on the event $\widetilde{\Upxi}$. Putting~\eqref{le dicho} back into~\eqref{eq:recursive3bis}, we obtain that
\begin{align}
\mathbb{E}[|P(m_t)|^{2D}]&\leq N^{4D\epsilon}\beta^{2D}|\alpha_2|^{2D} +N^{3D\epsilon}q_t^{-8D}+q_t^{-6D}|\alpha_2|^{4D} \notag \\
&\leq N^{6D\epsilon}\beta^{2D}|\alpha_2|^{2D}+N^{6D\epsilon}\beta^{4D},
\end{align}
for any small $\epsilon>0$, and large $D$, uniformly on $\widetilde{\mathcal{E}}$. Note that, for $z\in\mathcal{E}\backslash\widetilde{\mathcal{E}}$, it is direct to check the estimate $\mathbb{E}[|P(m_t)|^{2D}]\leq N^{6D\epsilon}\beta^{2D}|\alpha_2|^{2D}+N^{6D\epsilon}\beta^{4D}$. Applying a lattice argument and the Lipschitz continuity (see \emph{e.g.} the proof of Lemma~\ref{le new lemma 1}), we find from a union bound that for any small $\epsilon>0$ and large $D$ there exists an event $\Upxi$ with $\mathbb{P}(\Upxi)\geq 1-N^{-D}$ such that
\begin{align}\label{le upsi}
|\alpha_2(m_t-\widetilde{m}_t)+z(m_t-\widetilde{m}_t)^2|\leq N^{\epsilon}\beta\Lambda_{t}^{2}+N^{\epsilon}|\alpha_2|\beta+N^{\epsilon}\beta^2,
\end{align}
on $\Upxi$, uniformly on $\mathcal{E}$ for any sufficiently large $N$.
\end{proof}

We now prove Proposition \ref{prop:locallaw}, which will also imply Theorem \ref{thm:norm}.

\begin{proof}[Proof of Proposition \ref{prop:locallaw} and Theorem \ref{thm:norm}] Fix $t\in [0, 6 \log N]$. Let $\widetilde{m}_t$ be the solution $w_t$ in Lemma \ref{lemma:rho_t}. We proved statements $(1)$ and $(2)$ in Lemma~\ref{lemma:rho_t}, it hence remains to prove statement $(3)$ of Proposition~\ref{prop:locallaw}.

Recall that for fixed $E$ $\beta=\beta(E+\ii\eta)$ is a decreasing function of $\eta$, $\sqrt{\varkappa_{t}(E) +\eta}$ is an increasing function of $\eta$, and $\eta_0\equiv \eta_0(E)$ satisfies that $\sqrt{\varkappa(E)+\eta_0} = 10C_0 N^{\epsilon}\beta(E+\ii\eta_0)$. Further notice that $\eta_0(E)$ is a continuous function. We consider the subdomains of $\mathcal{E}$ defined by
\begin{align*}
\mathcal{E}_1&:=\{z=E+\ii\eta\in\mathcal{E} : \eta\leq \eta_0(E), 10N^{\epsilon}\leq N\eta\}, \\
\mathcal{E}_2&:=\{z=E+\ii\eta\in\mathcal{E} : \eta>\eta_0(E), 10N^{\epsilon}\leq N\eta\}.
\end{align*}
We consider the cases $z\in\mathcal{E}_1$, $z\in\mathcal{E}_2$ and $z\in\mathcal{E}\backslash(\mathcal{E}_1\cup\mathcal{E}_2)$, and split the stability analysis accordingly. Let $\Upxi$ be a high probability event such that~\eqref{le upsi} holds.

{\it Case 1:} If $z\in\mathcal{E}_1$, we note that $|\alpha_2|\leq C_0\sqrt{\varkappa(E)+\eta}\leq 10C_0^2 N^{\epsilon}\beta(E+\ii\eta)$. Then, we find that
\begin{align*}
|z|\Lambda_t^2 &\leq|\alpha_2|\Lambda_t + N^{\epsilon}\beta\Lambda_t^2 + N^{\epsilon}|\alpha_2|\beta + N^{\epsilon}\beta^2 \\ 
&\leq 10C_0^2 N^{\epsilon}\beta\Lambda_t + N^{\epsilon}\beta\Lambda_t^2 + (10C_0^2 N^{\epsilon}+1) N^{\epsilon}\beta^2,
\end{align*}
on $\Upxi$. Hence, there is some finite constant $C$ such that on $\Upxi$, we have $\Lambda_t\leq CN^{\epsilon}\beta$, $z\in\mathcal{E}_1$.

{\it Case 2:} If $z\in\mathcal{E}_2$, we obtain that
\begin{align}
|\alpha_2|\Lambda_t\leq (|z|+N^{\epsilon}\beta)\Lambda_t^2 + |\alpha_2|N^{\epsilon}\beta + N^{\epsilon}\beta^2,
\end{align}
on $\Upxi$. We then notice that $C_0|\alpha_2|\geq \sqrt{\varkappa_t(E)+\eta}\geq 10C_0 N^{\epsilon}\beta$, \emph{i.e.} $N^{\epsilon}\beta\leq |\alpha_2|/10$, so that
\begin{align}
|\alpha_2|\Lambda_t\leq (|z|+1)\Lambda_t^2 + (1+N^{-\epsilon})|\alpha_2|\beta,
\end{align}
on $\Upxi$, where we used that $N^{\epsilon}\beta\leq 1$. Hence, on $\Upxi$, either
\begin{align}
|\alpha_2|\leq 2(1+|z|)\Lambda_t \quad\textrm{or}\quad \Lambda_t\leq 3N^{\epsilon}\beta.
\end{align}
We now follow the dichotomy argument and the continuity argument used to obtain~\eqref{le dicho}. Since $3N^{\epsilon}\beta\leq |\alpha_2|/8$ on $\mathcal{E}_2$, by continuity, we find  that on the event $\Upxi$, $\Lambda_t\leq 3N^{\epsilon}\beta$ for $z\in\mathcal{E}_2$.

{\it Case 3:} For $z\in  \mathcal{E}\backslash(\mathcal{E}_1 \cup\mathcal{E}_2)$ we use that $|m'_t(z)|\leq \frac{\Im m_t(z)}{\Im z}, z\in\mathbb{C}^{+}$, since $m_t$ is a Stieltjes transform of a probability measure. Set $\widetilde{\eta}:=10N^{-1+\epsilon}$ and observe that
\begin{align}
|m_t(E+\ii\eta)|
&\leq \int_{\eta}^{\widetilde{\eta}}\frac{s\Im m_t(E+\ii s)}{s^2} \di s+ \Lambda_t(E+\ii\widetilde{\eta})+|\widetilde{m}_t(E+\ii \widetilde{\eta})|.
\end{align}
From the definition of the Stieltjes transform, it is easy to check that $s\rightarrow s\Im m_t(E+\ii s)$ is monotone increasing. Thus, we find that
\begin{align}
|m_t(E+\ii\eta)| &\leq\frac{2\widetilde{\eta}}{\eta}\Im m_t(E+\ii\widetilde{\eta})+\Lambda_t(E+\ii\widetilde{\eta}) + |\widetilde{m}_t(E+\ii\widetilde{\eta})| \notag \\
&\leq C\frac{N^{\epsilon}}{N\eta}\big(\Im \widetilde{m}_t(E+\ii\widetilde{\eta})+\Lambda_t(E+\ii\widetilde{\eta})\big) + |\widetilde{m}_t(E+\ii\widetilde{\eta})|,
\end{align}
for some $C$ where we used $\widetilde{\eta}=10N^{-1+\epsilon}$ to obtain the second inequality. Since $z=E+\ii\widetilde{\eta}\in\mathcal{E}_1 \cup \mathcal{E}_2$, we have $\Lambda_t(E+\ii\widetilde{\eta})\leq CN^{\epsilon}\beta(E+\ii\widetilde{\eta})\leq C$ on $\Upxi$. Using that $\widetilde{m}_t$ is uniformly bounded on $\mathcal{E}$, we get that, on $\Upxi$, $\Lambda_t\leq CN^{\epsilon}\beta$, for all $z\in \mathcal{E}\backslash(\mathcal{E}_1 \cup\mathcal{E}_2)$.

In sum, we get $\Lambda_t\prec\beta$ uniformly on $\mathcal{E}$ for fixed $t\in[0,6\log N]$. Choosing $t=0$, we have proved Theorem~\ref{thm:norm}. To prove that this bound holds for all $t\in[0,6\log N]$, we use the continuity of the Dyson matrix flow. Choosing a lattice $\mathcal{L}\subset[0,6\log N]$ with spacings of order $N^{-3}$, we find that $\Lambda_t\prec\beta$, uniformly on $\mathcal{E}$ and on $\mathcal{L}$, by a union bound. Thus, by continuity, we can extend the conclusion to all $t\in[0,6\log N]$ and conclude the proof of Proposition~\ref{prop:locallaw}.
\end{proof}

\subsection{Proof of Theorem \ref{thm:norm}} \label{sec6}
Theorem~\ref{thm:norm} is a direct consequence of the following result.
\begin{lemma}\label{lemma:norm}
Let $X_0$ satisfy Assumption \ref{assume} with $\phi>0$. Then,
\begin{align}\label{eq:normX'X}
\big| \|X_t^{\dagger}X_t \|-L_t \big|\prec\frac{1}{q_t^4} +\frac{1}{N^{2/3}},
\end{align}
uniformly in $t\in[0,6\log N]$.
\end{lemma}
The proof of Lemma~\ref{lemma:norm} is split into a lower and an upper bound. The lower bound is a direct consequence of the local law in Proposition~\ref{prop:locallaw}. The upper bound requires an additional stability analysis starting from the first inequality in~\eqref{eq:recursive3bis}. This time we capitalize on the fact that $\alpha_1(z)=\Im\widetilde{m}_t(z)$ behaves as $\eta/\sqrt{\varkappa_t(E)+\eta}$, for $E\ge L_+$, to get sharper estimates outside of the spectrum. Since the arguments for the lower and upper bounds are similar to the ones in~\cite{LeeSchnelli2016}, we postpone their proofs to the Appendix B of the Supplement~\cite{supplement}.

\section{Proof of Tracy--Widom limit for the largest eigenvalue} \label{sec:TW}
In this section, we prove the Theorem \ref{thm:TWlimit}, the Tracy--Widom limiting distribution of the largest eigenvalue. Following the idea from \cite{ErdosYauYin2012}, we consider the imaginary part of the normalized trace of the Green function $m\equiv m^{X^\dagger X}$ of $X^\dagger X$. For $\eta>0$, let
\begin{align}
\theta_\eta(y)=\frac{\eta}{\pi(y^2+\eta^2)},\quad(y\in\mathbb{R}).
\end{align}
It can be easily checked from the definition of the Green function that
\begin{align}
\Im m(E+\ii\eta)=\frac{\pi}{N}\Tr \theta_\eta(X^\dagger X-E).
\end{align}

The first proposition in this section shows how we can approximate the distribution of the largest eigenvalue by using the Green function. Recall that $L_+$ is the right endpoint of the deterministic probability measure in Theorem \ref{thm:locallaw}. 
\begin{proposition} \label{prop:Greenftn}
Let $X$ satisfy Assumption \ref{assume}, with $\phi>1/6$. Denote by $\lambda_1^{X^\dagger X}$ the largest eigenvalue of $X^{\dagger}X$. Fix $\epsilon>0$ and let $E\in\mathbb{R}$ be such that $|E-L_+|\leq N^{-2/3+\epsilon}$. Set $E_{+} :=L_++2N^{-2/3+\epsilon}$ and define $\chi_E:=\mathds{1}_{[E,E_{+}]}.$ Let $\eta_1:=N^{-2/3-3\epsilon}$ and $\eta_2:=N^{-2/3-9\epsilon}$. Let $K:\mathbb{R}\rightarrow[0,\infty)$ be a smooth function satisfying
\begin{align}
K(x)=\left\{\begin{array}{ll}
1&\textrm{if $|x|<1/3$} \\
0&\textrm{if $|x|>2/3$},
\end{array}\right.
\end{align}
which is a monotone decreasing on $[0,\infty)$. Then, for any $D>0$,
\begin{align}
\mathbb{E}[K(\Tr(\chi_E *\theta_{\eta_2})(X^\dagger X))]>\mathbb{P}(\lambda_1^{X^\dagger X}\leq E-\eta_1 )-N^{-D}
\end{align}
and
\begin{align}
\mathbb{E}[K(\Tr(\chi_E *\theta_{\eta_2})(X^\dagger X))]<\mathbb{P}(\lambda_1^{X^\dagger X}\leq E+\eta_1 )+N^{-D}
\end{align}
for $N$ sufficiently large, with $\theta_{\eta_2}$.
\end{proposition}
For the proof, we refer to Proposition 7.1 of \cite{LeeSchnelli2016}. We remark that the lack of the improved local law near the lower edge does not alter the proof of Proposition \ref{prop:Greenftn}. 

Next, we state the Green function comparison result for our model. We let $W^{G}$ be a $M\times N$ Gaussian matrix independent of $X$ and denote by $m^{G}\equiv m^{W^{G}}$ the normalized trace of its Green function.
\begin{proposition} \label{prop:Greencomparison}
Under the assumptions of Proposition~\ref{prop:Greenftn} the following holds. Let $\epsilon>0$ and set $\eta_0=N^{-2/3-\epsilon}$. Let $E_1, E_2\in\mathbb{R}$ satisfy $|E_1|,|E_2|\leq N^{-2/3+\epsilon}$. Consider a smooth function $F:\mathbb{R}\rightarrow\mathbb{R}$ such that
\begin{align}
\max_{x\in\mathbb{R}}|F^{(l)}(x)|(|x|+1)^{-C}\leq C,\quad(l\in\mathbf{[}1, 11\mathbf{]}).
\end{align}
Then, for any sufficiently small $\epsilon>0$, there exists $\delta>0$ such that
\begin{align}
\left| \mathbb{E}F\Big(\int^{E_2}_{E_1}\Im m(x+L_++\ii\eta_0)dx\Big) -\mathbb{E}F\Big(\int^{E_2}_{E_1}\Im m^{G}(x+\lambda_{+}+\ii\eta_0)dx\Big) \right|\leq N^{-\delta}
\end{align}
for large enough $N$.
\end{proposition}

Proposition \ref{prop:Greencomparison} directly implies Theorem \ref{thm:TWlimit}, the Tracy--Widom limit for the largest eigenvalue. A detailed proof is found, {\it e.g.}, with the same notation in~\cite{LeeSchnelli2016}, Section 7.

In the remainder of the section, we prove Proposition \ref{prop:Greencomparison}. We begin by the following application of the generalized Stein lemma.
\begin{lemma}
Fix $\ell \in \mathbb{N}$ and let $F\in C^{\ell+1} (\mathbb{R}; \mathbb{C}^+).$ Let $Y\equiv Y_0$ be a random variable with finite moments to order $\ell+2$ and let $W$ be a Gaussian random variable independent of $Y$. Assume that $\mathbb{E}[Y]=\mathbb{E}[W]=0$ and $\mathbb{E}[Y^2] =\mathbb{E}[W^2]$. Introduce
\begin{align}
Y_t := \mathrm{e}^{-t/2} Y_0 +\sqrt{1-\mathrm{e}^{-t}}W,
\end{align}
and let $\dot Y_t\equiv \di Y_t /\di t$.  Then,
\begin{align}\label{eq:genStein}
\mathbb{E}\Big[ \dot Y_t F(Y_t)\Big] = -\frac{1}{2} \sum_{r=2}^{\ell} \frac{\kappa^{(r+1)}(Y_0)}{r!} \mathrm{e}^{-\frac{(r+1)t}{2}}\mathbb{E} \big[F^{(r)}(Y_t)\big] + \mathbb{E}\big[\Omega_\ell (\dot Y_t F(Y_t))\big],
\end{align}
where $\mathbb{E}$ denotes the expectation with respect to $Y$ and $W$, $\kappa^{(r+1)}(Y)$ denotes the $(r+1)$-th cumulant of $Y$ and $F^{(r)}$ denotes the $r$-th derivative of the function $F$. The error term $\Omega_\ell$ in \eqref{eq:genStein} satisfies
\begin{align}
\big|\mathbb{E}\big[\Omega_\ell (\dot Y_t F(Y_t))\big]\big|
\leq C_\ell \mathbb{E}[|Y_t|^{|\ell+2}] \sup_{|x|\leq Q} |F^{(\ell+1)} (x)|+ C_\ell \mathbb{E}[|Y_t|^{\ell+2}\mathbf{1}(|Y_t|>Q)]\sup_{x\in \mathbb{R}}|F^{(\ell+1)} (x)|, 
\end{align}
where $Q> 0$ is an arbitrary fixed cutoff and $C_\ell$ satisfies $C_\ell \leq \frac{(C\ell)^{\ell}}{\ell !}$ for some numerical constant $C$.
\end{lemma}
\begin{proof}[Proof of Proposition \ref{prop:Greencomparison}] Fix a (small) $\epsilon>0$. Consider $x\in [E_1, E_2]$. For simplicity, let
\begin{align}
G\equiv G_t(x+L_t +\ii\eta_0), \quad m\equiv m_t(x+L_t+\ii\eta_0),
\end{align}
with $\eta_0=N^{-2/3-\epsilon}$, and define
\begin{align}
Y\equiv Y_t:=N\int^{E_2}_{E_1}\Im m(x+L_t +\ii\eta_0) \di x.
\end{align}
Note that $Y\prec N^{\epsilon}$ and $|F^{(l)}(Y)|\prec N^{C\epsilon}$ for $l\in[1,11]$. Recall from \eqref{eq:L_t} and \eqref{eq:L'_t} that
\begin{align*}
L_t&=\lambda_{+}+\frac{1}{\sqrt d}\big(1+\frac{1}{\sqrt d}\big)^2\mathrm{e}^{-t}s^{(4)}q_t^{-2}+O(\mathrm{e}^{-2t}q_t^{-4}), \\
\dot{L_t}&=-2\frac{1}{\sqrt d}\big(1+\frac{1}{\sqrt d}\big)^2\mathrm{e}^{-t}s^{(4)}q_t^{-2}+O(\mathrm{e}^{-2t}q_t^{-4}),
\end{align*}
with $q_t=\mathrm{e}^{t/2}q_0$. Let $z=x+L_t+\ii\eta_0$ and $G\equiv G(z)$. Differentiating $F(Y)$ with respect to $t$, we get
\begin{align}
\frac{\textrm{d}}{\textrm{d}t}\mathbb{E}F(Y) &= \mathbb{E}\Big[F'(Y)\frac{\textrm{d}Y}{\textrm{d}t}\Big]=\mathbb{E}\Big[F'(Y)\Im\int^{E_2}_{E_1}\sum^{N}_{i=1}\frac{\textrm{d}G_{ii}}{\textrm{d}t}\textrm{d}x\Big] \notag \\
&=\mathbb{E}\Big[F'(Y)\Im\int^{E_2}_{E_1}\Big(\sum_{i,j,\alpha}\dot{X}_{\alpha j}\frac{\partial G_{ii}}{\partial H_{\alpha j}}+\dot{L_t}\sum_{1\leq i,j\leq N} G_{ij}G_{ji}\Big)\textrm{d}x\Big], \label{eq:F_derivative}
\end{align}
where by definition
\begin{align}
\dot{X}_{\alpha j}\equiv(\dot{X}_{t})_{\alpha j}=-\frac{1}{2}\mathrm{e}^{-t/2}(X_0)_{\alpha j}+\frac{\mathrm{e}^{-t}}{2\sqrt{1-\mathrm{e}^{-t}}}W^{G}_{\alpha j}.
\end{align}
Thus, we find that
\begin{align}\label{eq:F'term1}
&\sum_{i,j,\alpha}\mathbb{E}\Big[\dot{X}_{\alpha j}F'(Y)\frac{\partial G_{ii}}{\partial X_{\alpha j}}\Big] = -2 \sum_{i,j,\alpha}\mathbb{E}\Big[\dot{X}_{\alpha j}F'(Y)G_{i\alpha}G_{ji}\Big] \notag \\
&=\frac{\mathrm{e}^{-t}}{N}\sum_{r=2}^{\ell}\frac{q_t^{-(r-1)}s^{(r+1)}}{r!}\sum_{1\leq i\leq j}\sum_{j,\alpha}\mathbb{E}[\partial_{\alpha j}^r (F'(Y)G_{i\alpha}G_{ji})] + O(N^{1/3 +C\epsilon}),
\end{align}
for $\ell=10$, where we use the short hand $\partial_{j\alpha}=\partial/\partial H_{j\alpha}$. Here, the error term $O(N^{1/3 +C\epsilon})$ in \eqref{eq:F'term1} corresponds to $\Omega_\ell$ in \eqref{eq:genStein}, which is $O(N^{C\epsilon} N^2 q_t^{-10})$ for $Y=H_{j\alpha}$.

We claim the following lemma which is proved in Appendix C of Supplements~\cite{supplement}.
\begin{lemma}\label{lemma:J_r}
For an integer $r\geq 2$, let
\begin{align}
J_r:=\frac{\mathrm{e}^{-t}}{N}\frac{q_t^{-(r-1}s^{(r+1)}}{r!}\sum_{i,j,\alpha}\mathbb{E}[\partial_{j\alpha}^r (F'(Y)G_{ij}G_{\alpha i})].
\end{align}
Then, for any $r\neq 3$,
\begin{align}
J_r=O(N^{2/3-\epsilon'}),
\end{align}
and
\begin{align}
J_3=\frac{2}{\sqrt d}\Big(1+\frac{1}{\sqrt d}\Big)^2 \mathrm{e}^{-t}s^{(4)}q_t^{-2}\sum_{i,j}\mathbb{E}[F'(Y)G_{ij}G_{ji}]+O(N^{2/3-\epsilon'}).
\end{align}
\end{lemma}

Assuming Lemma \ref{lemma:J_r}, we find that there exists $\epsilon'>2\epsilon$ such that, for all $t\in [0,6\log N]$,
\begin{align}
\sum_{i,j,\alpha}\mathbb{E}\Big[\dot{X}_{j\alpha} F'(Y)\frac{\partial G_{ii}}{\partial X_{j\alpha}}\Big]=-\dot{L_t}\sum_{i,j}\mathbb{E}[G_{ij}G_{ji}F'(Y)]+O(N^{2/3-\epsilon'}),
\end{align}
which implies that the right-hand side of \eqref{eq:F_derivative} is $O(N^{-\epsilon'/2})$. Integrating Equation \eqref{eq:F_derivative} from $t=0$ to $t=6 \log N$, we get
\begin{align*}
&\left|\mathbb{E}F \Big( N\int_{E_1}^{E_2} \Im m(x+L_t+\ii\eta_0) \textrm{d}x\Big)_{t=0} \right.\\
&\qquad \qquad \qquad \left.-\mathbb{E}F\Big(  N\int_{E_1}^{E_2} \Im m(x+L_t+\ii\eta_0) \textrm{d}x\Big)_{t=6 \log N}\right| \leq N^{-\epsilon'/4}.
\end{align*}
We remark that the comparison between $\Im m|_{t=6\log N}$ and $\Im m^{G}$ is trivial as in the Wigner-type case. Comparing $i$-th largest eigenvalues of $X^{\dagger}X$ and~$\lambda_i^{G}$, we find that
\begin{align}
\left| \Im m|_{t=6 \log N} - \Im m^{G} \right| \prec N^{-5/3}.
\end{align}
This completes the proof of desired proposition.
\end{proof}

\subsection*{Acknowledgements}

We thank Xiucai Ding for useful comments and suggestions. Jong Yun Hwang and Ji Oon Lee are supported in parts by the Samsung Science and Technology Foundation project number SSTF-BA1402-04. Kevin Schnelli is supported in parts by the G\"oran Gustafsson Foundation and the Swedish research council grant VR-2017-05195.
\subsection*{Supplementary Material}(\urlstyle{same}\url{http://mathsci.kaist.ac.kr/~jioon/sparse_covariance/Supplementary_arxiv.pdf})
In the supplementary material, we will provide the proofs of Lemmas \ref{lemma:recursive}, \ref{le new lemma 1}, \ref{lemma:norm} and \ref{lemma:J_r}. 

\end{document}